\documentclass[11pt,a4paper]{amsart}

\newcommand{\pres}[2]{\langle #1\:|\:#2 \rangle}

\usepackage{amssymb,amsmath,amsthm}
\usepackage{tikz-cd}
\usepackage{braids}
\usetikzlibrary{arrows}
\usepackage{enumitem} 
\usepackage{graphicx}
\usepackage[hidelinks]{hyperref}
\usepackage{mathtools}
\usepackage{graphicx}
\usepackage{caption}
\usepackage{subcaption}
\usepackage[english]{babel}
\usepackage{mathrsfs}
\usepackage{cite}
\usepackage{pifont}
\usepackage{mathtools}

\usepackage[pagewise]{lineno} 

\newtheorem{theorem}{Theorem}[section]
\newtheorem{lemma}[theorem]{Lemma}
\newtheorem{corollary}[theorem]{Corollary}
\newtheorem{proposition}[theorem]{Proposition}
\newtheorem*{claim*}{Claim}
\theoremstyle{definition}

\newtheorem{remark}[theorem]{Remark}

\newtheorem*{theorem*}{Theorem}
\newtheorem*{conjecture*}{Conjecture}

\newcommand{\B}{\mathbf{B}}
\newcommand{\PB}{\mathbf{PB}}

\newcommand{\A}{\mathcal{A}}

\newcommand{\Z}{\mathbb{Z}}

\makeatletter
\def\Ddots{\mathinner{\mkern1mu\raise\p@
\vbox{\kern7\p@\hbox{.}}\mkern2mu
\raise4\p@\hbox{.}\mkern2mu\raise7\p@\hbox{.}\mkern1mu}}

\makeatother

\DeclareSymbolFontAlphabet{\amsmathbb}{AMSb}

\DeclareMathOperator{\Mon}{Mon}
\DeclareMathOperator{\Sgp}{Sgp}

\DeclareMathOperator{\oSigma}{\overline{\Sigma}}
\DeclareMathOperator{\cA}{\mathcal{A}}

\DeclarePairedDelimiter\abs{\lvert}{\rvert}

\makeatletter
\let\oldabs\abs
\def\abs{\@ifstar{\oldabs}{\oldabs*}}

\newcounter{bobcomments}

\newcounter{cfcomments}

\begin{document}

\title[Membership problems in braid groups and Artin groups]{Membership problems in braid groups \\ and Artin groups}

\author[Gray]{Robert D. Gray }
\address{School of Engineering, Mathematics \& Physics, University of East Anglia, Norwich NR4 7TJ, England, UK}
\email{Robert.D.Gray@uea.ac.uk}

\author[Nyberg-Brodda]{Carl-Fredrik Nyberg-Brodda}
\address{June E Huh Center for Mathematical Challenges, Korea Institute for Advanced Study (KIAS), Seoul 02455, Republic of Korea} 
\email{cfnb@kias.re.kr}

\thanks{The first author is supported by the EPSRC Fellowship grant EP/V032003/1 ‘Algorithmic, topological and geometric aspects of infinite groups, monoids and inverse semigroups’. The second author is supported by the Mid-Career Researcher Program (RS-2023-00278510) through the National Research Foundation funded by the government of Korea, and by the KIAS Individual Grant MG094701 at Korea Institute for Advanced Study.}

\subjclass[2020]{20F36, 20F05 (primary), 20M05, 68Q70}

\keywords{Membership problems; Artin groups; braid groups.}

\date{\today}

\begin{abstract}
We study several natural decision problems in braid groups and Artin groups. We classify the Artin groups with decidable submonoid membership problem in terms of the non-existence of certain forbidden induced subgraphs of the defining graph. 
Furthermore, we also classify the Artin groups for which the following problems are decidable: 
the rational subset membership problem, semigroup intersection problem, and the fixed-target submonoid membership problem.
In the case of braid groups our results show that 
the submonoid membership problem, and each and every one of these problems, is decidable in the braid group $\B_n$ if and only if $n \leq 3$, which answers an open problem of Potapov (2013). 
Our results also generalize and extend results of Lohrey \& Steinberg (2008) who classified right-angled Artin groups with decidable submonoid (and rational subset) membership problem.      
\end{abstract}
%

\maketitle

\vspace{-8mm}

\section{Introduction}

\noindent Algorithmic problems are of central importance in the study of braid groups, and more generally of Artin groups.  
For braid groups, the word and conjugacy problems are both decidable, and they have been studied extensively from many different viewpoints in numerous important papers; see e.g. 
\cite{artin1947theory, 
birman1998new,
dehornoy1997fast, 
elrifai1994algorithms,
garside1969braid, 
garber2002new, 
gonzalez2014twisted
}.
For more background on braid groups, including the history and motivation for their study, including connections with representation theory, algebraic geometry and topology, and mathematical physics, we refer the reader to the book \cite{kassel2008braid}.

Artin groups, which are a natural generalization of braid groups closely related to Coxeter groups, first appeared in the literature in the papers \cite{deligne1972immeubles,brieskorn1972artin}, in which the word problem is also considered. Even after fifty years of intensive study, Artin groups remain a highly mysterious class and in general very little is known about the algorithmic properties of Artin groups; see e.g. the survey articles \cite{Birman1974, godelle2012basic} and \cite{mccammond2017mysterious}.  In particular, both the word and conjugacy problems remain open in general for Artin groups. A lot of work has been done in this area, and it continues to be a highly active area of research. It is also known that finite type Artin groups (which includes all braid groups) have nice algorithmic properties, and in particular have decidable word and conjugacy problem \cite{deligne1972immeubles,brieskorn1972artin}; it is also known that they are biautomatic
\cite{
charney1992artin,
charney1995geodesic
} which gives very efficient solutions to the word and conjugacy problems in these cases. Additional work on the word and conjugacy problems in finite type Artin groups include \cite{BirmanKoLee1998, BirmanGebhardt2007, BirmanGebhardt2007b, BirmanGebhardt2008, GebhardtGonzalez2010, GebhardtGonzalez2010b}. Other papers on Artin groups and their algorithmic properties include e.g. 
\cite{
brady2000three,
Crisp2009, 
haettel2023new,
Hermiller1999,
huang2020large,
holt2015conjugacy,
holt2013shortlex,
holt2012artin,
mccammond2017artin,
Picantin2001
}; for some recent progress, we also refer the reader to \cite{BlascoGarcia2022, BlascoGarciaCumplidoHoltMorrisWrightRees2024, blasco2022word} and references therein.
Finitely generated submonoids of Artin groups play an important part in the study of their structure.
Paris \cite{paris2002artin} proved that the positive submonoid, called the Artin monoid, naturally embeds in its Artin group. 
For some particular classes of Artin groups (e.g. those of spherical type) information about the Artin monoid can be used to understand the Artin group. Artin monoids have been studied e.g. in \cite{brieskorn1972artin, boyd2024artin, michel1999note}.  
For Artin groups in general there is currently no known way to translate information from the Artin monoid to the Artin group; see the introduction of \cite{boyd2024artin} for further discussion of this.   

When studying finitely generated submonoids of Artin groups certain natural algorithmic problems naturally arise. The most fundamental of these is the \textit{membership problem}, which given a finitely generated submonoid of the group asks whether there is an algorithm that, given an element of the group (as a word over the generators), can decide whether that element belongs to the submonoid. 
This generalizes the subgroup membership problem, also sometimes called the generalized word problem, which asks about deciding membership in a finitely generated subgroup.   

In this article we will study the membership problem in finitely generated submonoids, and more generally rational subsets, of Artin groups. We will also investigate a range of other related membership-type decision problems including the fixed-target submonoid membership problem and semigroup intersection problem.   
Further motivation for studying the submonoid membership problem in Artin groups comes from the fact that the word problem for Artin groups can be reduced to deciding membership in particular finitely generated submonoids. For example, it is not hard to see that for any Artin group if membership in the Artin submonoid of that Artin group is decidable then it follows that the Artin group has decidable word problem. Indeed, in this case to decide whether a given word equals the identity one can first check whether it belongs to the Artin submonoid. If it does not belong to the submonoid, then it is not equal to the identity element; and if it does, then we can use the fact that the Artin monoid (trivially) has decidable word problem to test whether the word equals the identity in the Artin monoid, and hence in the group. 

Membership problems in braid and Artin groups have already received some attention in the literature. Makanina \cite{Makanina1981} proved that $\B_n$ has undecidable subgroup membership problem for $n \geq 5$, and Potapov \cite{Potapov2013Composition} proved that the submonoid membership problem in $\B_3$ is $\operatorname{NP}$-hard when $n =3$. He left as an open question whether the submonoid membership problem is decidable in the braid group $\B_4$. In a later paper of Ko \& Potapov \cite{KoPotapov2017} the submonoid membership problem in $\B_4$ is again mentioned as an open problem, and in relation to this problem they observe that by a result of \cite{Akimenkov1991} there is no embedding from a set of pairs of words into $\B_4$, leading them to speculate that 
the submonoid membership problem might be decidable for $\B_4$ since 
the proof that $\B_5$ has undecidable submonoid membership problem, and many other undecidability results for $\B_5$,    
essentially rely on finding an embedding from a set of pairs of words into $\B_5$; see 
\cite{Makanina1981, Potapov2013Composition}.

The first result we will prove in this paper is that, contrary to their expectation, the braid group $\B_4$ in fact has undecidable submonoid membership problem. More precisely we shall show that $\B_4$ contains a fixed finitely generated submonoid in which membership is undecidable. Combined with Potapov's results in \cite{Potapov2013Composition}, this completes the classification of braid groups with decidable submonoid membership problem as being precisely those braid groups $\B_n$ where $n \leq 3$.     
The key new tool that we use to establish this result for braid groups is to make use of the classification of right-angled Artin groups (RAAGs) with decidable submonoid membership problem \cite{Lohrey2008} due to Lohrey \& Steinberg, combined with a result of Droms, Lewis \& Servatius about embedding particular RAAGs into $\B_4$; see Theorem~\ref{thm:BraidClassification}.
This result also shows that the same classification result also determines the braid groups with decidable rational subset membership problem.  
  
The families of braid groups and of RAAGs, both belong to the larger class of Artin groups. With the classification for braid groups with decidable submonoid membership problem 
described in the previous paragraph 
in hand, and also the corresponding result for RAAGs from \cite{Lohrey2008}, the next natural question is whether we can classify the Artin groups with decidable submonoid membership problem. It turns out that we can. In Theorem~\ref{thm_Artin_in_general} we shall
completely classify the Artin groups with decidable submonoid membership problem in terms of the non-existence of certain forbidden induced subgraphs of the defining graph. 

One consequence of the results we prove in this paper is to show that for Artin groups there is a close connection between the submonoid membership problem and another well-studied notion in the literature called subgroup separability (also called LERF). It will follow from our results, together with the results from \cite{AlmeidaLima2021}, 
that the class of Artin groups with decidable submonoid membership problem coincides exactly with the class of Artin groups that are subgroup separable. 
We stress that for finitely presented groups there is in general  no implication between these properties in the sense that 
a subgroup separable group need not have decidable submonoid membership problem, and vice versa.  
Indeed, finitely generated nilpotent groups are subgroup separable by Mal'cev 1948 
(see  
\cite[Exercise 11, Chapter 1]{segal2005polycyclic})
but there are nilpotent groups with undecidable submonoid membership problem; see e.g. \cite{Romankov2022}.   
For the other direction observe that the Baumslag-Solitar group $\operatorname{BS}(1,2)$ has decidable submonoid membership problem by \cite{Cadilhac2020} but it is not subgroup separable by results from 
\cite{blass1974application} (see also \cite[Proposition~1]{raptis1996subgroup}).   

In addition to the submonoid membership problem, there are a range of other natural membership-type problems in groups that have been introduced and investigated including: rational subset membership, semigroup intersection, the group problem, and the identity problem (see below for definitions of these notions). 
The study of these decision problems has motivation coming from both mathematics and theoretical computer science, and is a burgeoning field of research with a lot of recent intensive activity; see e.g. the recent survey articles \cite{Dong2023, lohrey2024membership} and also \cite{Lohrey2015}. 
For a helpful diagram showing how all these decision problems relate to each other we refer the reader to \cite[Fig. 1]{Dong2023}. 
In addition to the submonoid membership problem, we shall also consider these other decision problems listed above for Artin groups. We will show that our classification of Artin groups with decidable submonoid membership problem also gives a classification of Artin groups with decidable rational subset membership problem, and with decidable semigroup intersection problem; see Corollary~\ref{corol_Artin_in_general_OtherAlgProblems}. 

We do not currently know if this extends to also classify the Artin groups with decidable group problem or identity problem. In this direction in Theorem~\ref{thm_Artin_in_general_OtherAlgProblems} we identify some families of Artin groups for which the identity problem and group problem are undecidable, and in the remaining cases of our classification we show that a related problem is undecidable: the fixed-target submonoid membership problem (that we shall define). These results are established by proving several new undecidability results for the right-angled Artin group $A(P_4)$, where $P_4$ denotes the path graph on four vertices, in Section~\ref{Sec:AP4-undec}.  

In addition to presenting these algorithmic results on braid and Artin groups, another aim of this paper is to bring to the attention to those working on group and semigroup membership problems this method of embedding $A(P_4)$ as an approach to proving undecidability results for groups that are known not to contain $F_2 \times F_2$.  Indeed, it was recently proved in \cite[Corollary~6.4]{foniqi2023membership} that just embedding the trace monoid $T(P_4)$ into a group is enough to yield undecidability of membership in rational subsets. We believe that there are likely many more natural situations when pairs of words cannot be embedded into a group, but $A(P_4)$ or $T(P_4)$ can, giving rise to undecidability results. 
This is, for example, exactly what happens in the case of one-relator groups and monoids; see \cite{gray2020undecidability, foniqi2023membership} and also the related work \cite{minasyan2024right}.

The main results of the article are the following:
\begin{itemize}
\item We classify precisely when the submonoid (and rational subset) membership problem is decidable in braid groups (Theorem~\ref{thm:BraidClassification}). 
\item We show that the fixed-target submonoid membership problem is undecidable in $\B_4$ (Theorem~\ref{Thm:Undecidability-b4}).
\item We classify precisely when the submonoid membership problem is decidable in Artin groups in terms of the non-existence of certain forbidden induced subgraphs of the defining graph (Theorem~\ref{thm_Artin_in_general}).
\item We show that the same classification also applies for decidability of other decision problems (the rational subset membership and the semigroup intersection problem) in Artin groups (Corollary~\ref{corol_Artin_in_general_OtherAlgProblems}).
\item We identify some families of Artin groups for which the identity problem and group problem are undecidable (Theorem~\ref{thm_Artin_in_general_OtherAlgProblems}). 
\end{itemize}

\subsection*{Acknowledgements} We wish to thank the anonymous referees for their careful reading of the article and useful comments.

\section{Background}\label{Sec:background}

\subsection{Monoid presentations and rational subsets}

A finite set $A = \{ a_1, \dots, a_n \}$ is called an \textit{alphabet}. The free monoid on $A$, i.e.\ the set of all finite-length words on $A$ with the operation concatenation, is denoted $A^\ast$. Let $A^{-1}$ denote a set in bijective correspondence with $A$ via $a_i \mapsto a_i^{-1}$, and such that $A \cap A^{-1} = \varnothing$. Then the free group $F_A$ on $A$ is the set of all freely reduced words over $A \cup A^{-1}$, i.e.\ words not containing subwords of the form $xx^{-1}$ or $x^{-1}x$ for any $x \in A$. If $|A| = n$, then the free group of rank $n$ is denoted $F_n$. For a set of relations $R \subseteq F_A \times F_A$, a \textit{group presentation} $\pres{A}{R}$ is the quotient of $F_A$ by the least normal subgroup containing $uv^{-1}$ for all $(u, v) \in R$. For more details on combinatorial group theory, we refer the reader to \cite{Lyndon1977}. 

A subset of $A^\ast$ is called a \textit{language}. A \textit{finite automaton} $\cA$ (over the alphabet $A$) consists of a finite set $Q$ of \textit{states}, a finite set $\Delta \subseteq Q \times A \times Q$ of \textit{transitions}, a distinguished \textit{start state} $q_0 \in Q$, and a set $F \subseteq Q$ of \textit{final states}. We will consider $\cA$ as a directed graph, with vertex set $Q$ and an edge $q \xrightarrow{a} q'$ whenever $(q, a, q') \in \Delta$. The \textit{language accepted by $\cA$}, denoted $L(\cA)$, is the set of all words $w \equiv a_1 a_2 \cdots a_n$, where $x_i \in A$, for which there exist $q_1, \dots, q_n \in Q$ such that for all $1 \leq i \leq n$, we have $(q_{i-1} \xrightarrow{a_i} q_i) \in \Delta$, i.e.\ such that $w$ labels a directed path from $q_0$ to some state in $F$. Any language $R \subseteq A^\ast$ such that $R = L(\cA)$ for some finite automaton $\cA$ is called \textit{regular}. To simplify some technical steps, we will also allow $\varepsilon$-\textit{transitions}, being edges of the form $q_i \xrightarrow{\varepsilon} q_{j}$. This means that we can pass from state $q_i$ to $q_{j}$ without reading any letter. It is well-known that the class of regular languages is the same as the class of languages accepted by finite automata allowing $\varepsilon$-transitions. 

For a monoid $M$ and a subset $X \subseteq M$, we denote by $\Mon(X)$ the \textit{submonoid generated by $X$}. The subsemigroup generated by $X$ is denoted $\Sgp(X)$. The set of \textit{rational subsets} of $M$ is the least subset of $2^M$ such that (i) any finite subset is rational; (ii) the union, product, resp.\ the submonoid generated by a rational subset is again rational. Let $A$ be a finite generating set of a monoid $G$, and let $\pi \colon A^\ast \to G$ be a surjective homomorphism. Then it is classical to show (see \cite[Proposition~1]{Lohrey2015}) that $L \subseteq G$ is rational if and only if there is a finite automaton $\cA$ on $A$ such that $L = \pi (L(\cA))$. In particular, this definition does not depend on the choice of finite generating set $A$ or homomorphism $\pi$. If $X \subseteq A^\ast$, then if the context is clear we will often write $\Mon(X)$ for $\Mon(\pi(X))$, and analogously for $\Sgp(X)$. 

\subsection{Membership problems}\label{Subsec:decision problems}

There are many classical decision problems in algebra, with the most fundamental example being the \textit{word problem}, which is the problem of deciding whether or not two words over some generating set of a (semi)group represent the same element. Another broad class of decision problems that has distinguished itself as being of particular importance is the class of \textit{membership problems}. In its most general form, a membership problem for a monoid $M$ takes as input a subset of $M$ (e.g.\ a submonoid), and an element of $M$ (represented by a word), and decides whether or not the element belongs to the subset. By fixing parts of the input in this general membership problem, we obtain more restricted problems: if the input word is fixed, we obtain a \textit{fixed-target} membership problem; and if the subset is fixed, then we obtain a \textit{non-uniform} membership problem (otherwise, it is \textit{uniform}). Of course, if both the input word and the subset are fixed, then the membership problem has no input and is trivially decidable.

There are many actively studied decision problems which fall under these umbrella terms. Let $G$ be a group generated by a finite set $A$, with a surjective homomorphism $\pi \colon A^\ast \to G$. The \textit{(uniform) submonoid membership problem} for $G$ asks, on input a finite set $X \subseteq A^\ast$ of words, and a word $w \in A^\ast$, whether or not $\pi(w) \in \Mon(X)$. The \textit{non-uniform} submonoid membership is the same problem, except the finite subset $X$ is considered fixed and not part of the input; similarly, the \textit{fixed-target} submonoid membership problem fixes the word $w$, and only has $X$ as part of the input. This problem will thus be referred to as the ``fixed-target submonoid membership problem for $w$ in $M$''. All the above terminology is retained if instead of submonoids we consider membership in \textit{subsemigroups} $\operatorname{Sgp}(X)$. A particular case of the fixed-target subsemigroup problem is when the fixed target is $1$, i.e.\ when we are given a set $X$ and asked to decide if $1 \in \Sgp(X)$. This problem is known in the literature, e.g.\ \cite{Dong2023}, as the \textit{identity problem}\footnote{The reader is warned that the terminology ``identity problem'' is somewhat unfortunate, as this is also the classical name for the word problem, used e.g.\ by Dehn \cite{Dehn1911}.}. Analogously, the \textit{(uniform) rational subset membership problem} for $G$ asks, on input a finite automaton $\cA$ over $A$ and a word $w \in A^\ast$, whether or not $\pi(w) \in \pi(L(\cA))$. The \textit{non-uniform} version of this problem is the same, except the automaton $\cA$ is fixed, and not part of the input. In principle, one could consider the \textit{fixed-target} rational subset membership problem, fixing a target element $w$, and leaving $\cA$ as the only input. However, this can easily seen to be equivalent to the (uniform) rational subset membership problem; we thank C.\ Bodart for pointing out this fact to us. 

Another important membership-type problem that will be studied herein is the \textit{subsemigroup intersection problem}, which asks, on input two finite subsets $X_1, X_2 \subseteq A^\ast$, whether or not $\operatorname{Sgp}(X_1) \cap \operatorname{Sgp}(X_2) = \varnothing$ in $G$. This also has more and less uniform, as well as fixed-target, variants. For example, one might fix one of the subsets $X_i$ (making the problem less uniform), or ask if some \textit{fixed} element sits in the intersection (making the problem of a fixed-target type). We remark that if a group $G$ has decidable rational subset membership problem, then we can also solve the identity problem and the subsemigroup intersection problem in $G$. Indeed, the first is immediate, and the latter follows from the fact that $S_1S_2^{-1}$ is a rational subset of $G$ for any two finitely generated subsemigroups $S_1, S_2$. But $S_1 \cap S_2 \neq \varnothing$ if and only if $1 \in S_1 S_2^{-1}$, and since we decide the latter, we can also decide the former. Finally, another decision problem is the \textit{group problem}, which takes as input a finite set $X$ of words and decides whether or not $\Sgp(X)$ is a subgroup of $G$. 

Decidability of any of the above decision problems does not depend on the finite generating set chosen. Other decision problems, e.g.\ the subgroup membership problem, are defined analogously. Decidability of the rational subset membership problem is preserved by taking free products and finite extensions of groups, see \cite[Theorem~15]{Lohrey2015}. On the other hand, decidability of the submonoid membership is \textit{not} preserved by taking free products, as recently shown by Bodart \cite{Bodart2024}.

\begin{remark}\label{rmk:preservation}
The decision problems discussed in this section behave well when passing from a finitely generated group to a finitely generated subgroup. More precisely, let $G$ be a finitely generated group and $H \leq G$ a finitely generated subgroup, and suppose that $\mathcal{P}$ is any of the following decision problems:  submonoid membership problem, subgroup membership problem, rational subset membership problem,  fixed-target submonoid membership problem with a fixed target element $\gamma$, subsemigroup intersection problem, group problem, and identity problem. Then it is straightforward to see that if $\mathcal{P}$ is decidable in $G$ then $\mathcal{P}$ is also decidable in $H$. This is an important tool for showing various problems are undecidable in a group, by embedding a suitable group for which we know the problem is undecidable.  
  \end{remark}

\subsection{Artin and braid groups}\label{Subsec:Artin-braid-groups}

Let $\Gamma$ be a finite simplicial graph with edges labelled by natural numbers greater than or equal to $2$. Then the Artin group $A(\Gamma)$ is the group defined by the presentation with generating set the set $V\Gamma$ of vertices of $\Gamma$ and a defining relation   
\[
\underbrace{abab\ldots}_{\mbox{$m$-factors}} = \underbrace{baba\ldots}_{\mbox{$m$-factors}} 
\]      
for each edge that connects vertices $a$ and $b$ and is labelled $m$. We call a subgraph $\Delta$ of $\Gamma$ \emph{induced} if for any two vertices in $\Delta$ if they are adjacent in $\Gamma$ then they are also adjacent in $\Delta$. 

In this article there are two particular families of Artin groups that will play an important role, namely \textit{braid groups} and \textit{right-angled} Artin groups. For $n \geq 3$, the $n$\textit{-strand braid group} $\B_n$ is defined to be the group
\begin{equation}\label{Eq:braid-group-presentation}
\B_n := \pres{\sigma_1, \dots, \sigma_{n-1}}{\sigma_i \sigma_{i+1} \sigma_i = \sigma_{i+1} \sigma_i \sigma_{i+1}, \: \sigma_i \sigma_j = \sigma_j \sigma_i},
\end{equation}
with the first set of relations being taken for all $1 \leq i <n-1$, and the second set only being taken for those $i, j$ such that $|i-j| \geq 2$. There is a natural pictorial description of any element (braid) of a braid group, which also gives rise to a simple way to multiply braids together; see \cite[\S 1.2.2]{kassel2008braid}. 
Clearly the $n$-strand braid groups $\B_n$ are all Artin groups. For example, the braid group $\B_4$ is the Artin group $A(\Gamma)$ where $\Gamma$ is a triangle with edges labelled $(2,3,3)$.     

A \emph{right-angled Artin group} is an Artin group with the property that every edge in its defining graph $\Gamma$ is labelled by the number $2$. In this case, we omit the edge labels: given a finite, undirected, simple graph $\Gamma$ with vertices $V$ and edges $E$, the \textit{right-angled Artin group} $A(\Gamma)$ is thus the group with finite presentation 
\begin{equation}\label{Eq:raag-presentation}
A(\Gamma) := \pres{V}{v_i v_j = v_j v_i \: \text{ if and only if $(v_i, v_j) \in E$}}.
\end{equation}
Due to the form of their defining relations, right-angled Artin groups are sometimes also called \textit{partially commutative groups}. For a general introduction to right-angled Artin groups, see \cite{Charney2007}. 
Given a graph $\Gamma$ the monoid $T(\Gamma)$ with the same presentation as in \eqref{Eq:raag-presentation}, but viewed as a monoid presentation, is called the \emph{trace monoid} with defining graph $\Gamma$. It is known \cite{paris2002artin} that the trace monoid naturally embeds in the corresponding right-angled Artin group.     

\section{Braid groups and right-angled Artin groups}\label{Sec:AP4-undec}

As explained in the introduction, Potapov \cite{Potapov2013Composition} in his study of the membership problem in braid groups left as an open problem the case of the submonoid membership problem for $\B_4$, whereas the problem is decidable in $\B_3$ and undecidable in $\B_n$ for $n \geq 5$. We begin this section by presenting a result that completes the classification of braid groups with decidable submonoid membership problem, and shall also observe that the result also characterizes those braid groups with decidable rational subset membership problem. 

\begin{theorem}\label{thm:BraidClassification} 
The braid group $\B_n$ has 
decidable submonoid membership problem if and only 
$n \leq 3$.
Furthermore, 
if $n \leq 3$ then $\B_n$ has decidable rational subset membership problem, while  
if $n \geq 4$ then $\B_n$ contains a fixed finitely generated submonoid in which membership is undecidable.      
  \end{theorem}

Theorem~\ref{thm:BraidClassification} will be proved by applying several results from the literature on right-angled Artin groups and their embeddings  into braid groups. First, we have the following result of Lohrey \& Steinberg in which they classify the right-angled Artin groups with decidable submonoid, and rational subset, membership problem.

\begin{theorem}[Lohrey \& Steinberg {\cite[Theorem~1 and Corollary~3 ]{Lohrey2008}}]\label{Thm:LS} 
A right-angled Artin group $A(\Gamma)$ has decidable submonoid membership problem if and only $\Gamma$ does not contain $C_4$ or $P_4$ as an induced subgraph.      
Furthermore, if $\Gamma$ does not contain $C_4$ or $P_4$, then $A(\Gamma)$ has decidable rational subset membership problem, while if $\Gamma$ does contain $C_4$ or $P_4$ then $A(\Gamma)$ contains a fixed finitely generated submonoid in which membership is undecidable.     
  \end{theorem}

Here, and throughout this present article, $C_4$ denotes the cycle graph on four vertices, and $P_4$ denotes the path graph on four vertices. As mentioned in the introduction, there is no embedding of pairs of words, i.e.\ the trace monoid $T(C_4)$, into $\B_4$, and hence the group $A(C_4) \cong F_2 \times F_2$ also does not embed into the braid group $\B_4$. However, the following result of Droms, Lewin \& Servatius shows that the group $A(P_4)$ does embed into $\B_4$.   

\begin{theorem}[Droms, Lewin \& Servatius {\cite[Corollary~1]{Droms1991}}]\label{Thm:Droms-Ap4-embeds}
The subgroup of $\B_4$ generated by $s_2^2, (s_2 s_3 s_2)^2, s_3^2, s_1^2$ is isomorphic to $A(P_4)$. 
\end{theorem}

\begin{proof}[Proof of Theorem~\ref{thm:BraidClassification}.]
The group $\B_3$ is isomorphic to the torus knot group $\pres{x,y}{x^2 = y^3}$ (see e.g. \cite[Proof of Theorem 5.1]{CrispParis2005}), which is known to have a finite index subgroup isomorphic to $F_n \times \Z$ (see \cite{Niblo2001}).  
The group $F_n \times \Z$ has decidable rational subset membership problem by \cite[Theorem 6.5]{Kambites2007}. Since having decidable rational subset membership problem is preserved by taking finite index extensions (see \cite[Section~5]{Lohrey2015}) it follows that $\B_3$ also has decidable rational subset membership problem, and hence so does $\B_n$ for all $n \leq 3$.     

If $n \geq 4$ then by Theorem~\ref{Thm:Droms-Ap4-embeds} of Droms, Lewis \& Servatius the group $\B_n$ contains a subgroup isomorphic to the RAAG $A(P_4)$, and by Lohrey \& Steinberg's Theorem~\ref{Thm:LS}
the group $A(P_4)$ contains a fixed finitely generated submonoid in which membership is undecidable. Hence, together with Remark~\ref{rmk:preservation}, it follows that for $n \geq 4$ the braid group $\B_n$ contains a fixed finitely generated submonoid in which membership is undecidable.
\end{proof}

In Section~\ref{Sec:ArtinGroups} we will extend Theorem~\ref{thm:BraidClassification} to general Artin groups, giving the result Theorem~\ref{thm_Artin_in_general} which provides a common generalization of both Theorem~\ref{thm:BraidClassification} and the Theorem~\ref{Thm:LS} of Lohrey \& Steinberg for RAAGs.  

Before turning our attention to general Artin groups, we shall investigate some other decision problems for braid groups.  
Since $\B_3$ has decidable rational subset membership problem it follows that the submonoid membership problem, the subsemigroup intersection problem, group problem, identity problem, and subgroup membership problems are all also decidable in $\B_3$. 
By a classical result of Mikhailova \cite{Mikhailova1958} $A(C_4) \cong F_2 \times F_2$ has undecidable subgroup membership problem (in fact contains a fixed finitely generated subgroup in which membership is undecidable), while Bell \& Potapov \cite{Bell2010} have proved that $A(C_4) \cong F_2 \times F_2$ has undecidable identity problem. It follows from this that the identity problem, group problem, submonoid membership problem, and the semigroup intersection problem are all undecidable in $A(C_4) \cong F_2 \times F_2$. But Makanina \cite{Makanina1981} proved that there exists an embedding $F_2 \times F_2 \leq \B_n$ when $n \geq 5$, and it follows that all of these problems are also undecidable in in $\B_n$ with $n \geq 5$; for the submonoid membership problem this is immediate, and for the identity and group problems this was proved by Potapov in \cite[Theorem 14]{Potapov2013Composition}.    

This leaves the question of decidability of each of:  
the identity problem, group problem, and the semigroup intersection problem in $\B_4$. Since $F_2 \times F_2$ does not embed into $\B_4$ one cannot get undecidability results in the same way as for $\B_n$ with $n \geq 5$. However, we know that $A(P_4)$ does embed in $\B_4$, and in the rest of this section we explore the extent to which this can utilized to gain insights into which of these problems is decidable for the four-strand braid group $\B_4$.    

\subsection{A connection between rational subsets and submonoids}

The rational subset membership problem is undecidable in the right-angled Artin group $A(P_4)$. In this section we will build on ideas of Lohrey \& Steinberg \cite{Lohrey2008}  (cf.\ also \cite[Lemma~6.8]{foniqi2023membership}) to generalize this result somewhat for our purposes.  
The proof of the following lemma follows similar lines of the proof of \cite[Lemma~11]{Lohrey2008}.

\begin{lemma}\label{Lem:reachability-for-RAT1-reduces-to-SMP-reach-xY}
Let $G$ be any finitely generated group, and let $F$ be the free group with basis $\{ x, y, z \}$. Then the following hold: 
\begin{enumerate}[label=(\roman*)]
\item If the fixed-target submonoid membership problem for $xy^{-1}$ in $G \ast F$ is decidable, then the rational subset membership problem in $G$ is decidable.
\item If the semigroup intersection problem is decidable in
$G \ast F$ 
with respect to the subsemigroup $\operatorname{Sgp}(xy^{-1})$, then the rational subset membership problem in $G$ is decidable.
\end{enumerate}
\end{lemma}
\begin{proof}
Let $A$ be a finite generating set for $G$, and let $X = \{ x, y, z\}$, assuming without loss of generality that $A \cap X = \varnothing$. Let $\Sigma = A \cup X$, and let $\oSigma = \Sigma \cup \Sigma^{-1}$. Let $\pi \colon \oSigma^\ast \to G \ast F$ be the canonical surjective homomorphism. Let $R \subseteq \oSigma^\ast$ be any regular language, and let $\cA$ be a finite state automaton over $\oSigma$ such that $L(\cA) = R$, allowing $\varepsilon$-edges; in particular, we may assume without loss of generality that $\cA = (Q, \oSigma, \delta, q_0, q_f)$ has a single initial state $q_0 \in Q$ and a single accepting state $q_f \in Q$, where $q_0 \neq q_f$. Let
\[
\widetilde{Q} = \{ x, y, z^i x z^{-i} \: (1 \leq i \leq |Q|-2) \}
\]
which is a finite set of size $|Q|$. We fix a bijection between the two sets $Q$ and $\widetilde{Q}$ such that $q_0 \mapsto x$ and $q_f \mapsto y$. For $q \in Q$, let $\widetilde{q}$ denote the image of $q$ in $\widetilde{Q}$. Clearly, $\widetilde{Q}$ generates a free subgroup of $F$ of rank $|Q|$.

The idea is now to encode the transitions of the automaton $\cA$ into elements of $G \ast F$. To do this, we define the (finite) set 
\[
\Delta = \{ \widetilde{p} \sigma \widetilde{q}^{-1} \mid p \xrightarrow{\sigma} q \text{ is a transition in $\cA$} \}
\]
and we will prove that 
\begin{equation}\label{Eq:delta-simulates-A}
1 \in \pi(R) \iff \pi(xy^{-1}) \in \operatorname{Mon}(\Delta).
\end{equation}
Thus, proving that \eqref{Eq:delta-simulates-A} holds clearly suffices to establish part (i) of the lemma, for deciding whether $1 \in \pi(R)$ for any regular language $R$ is equivalent to solving the rational subset membership problem, as $u \in \pi(R)$ if and only if $1 \in u^{-1}\pi(R)$ where $u^{-1}\pi(R)$ is a rational subset. Similarly, we will also prove that 
\begin{equation}\label{Eq:semigroup-intersection-eq}
\operatorname{Sgp}(xy^{-1}) \cap \operatorname{Sgp}(\Delta) \neq \varnothing \iff \pi(xy^{-1}) \in \operatorname{Mon}(\Delta).
\end{equation}
This second part clearly implies part (ii) of the lemma. Indeed, taking equations \eqref{Eq:delta-simulates-A} and \eqref{Eq:semigroup-intersection-eq} together it follows that deciding    
$\operatorname{Sgp}(xy^{-1}) \cap \operatorname{Sgp}(\Delta) \neq \varnothing$ is equivalent to deciding 
$1 \in \pi(R)$ which, as explained in the previous paragraph, is equivalent to solving the rational subset membership problem

To prove \eqref{Eq:delta-simulates-A} and \eqref{Eq:semigroup-intersection-eq}, and thereby establish the lemma, we will rely on the following claim, which generalizes part of  
\cite[Claim~1, proof of Lemma~11]{Lohrey2008}.
We shall see that this claim will be sufficient to establish both \eqref{Eq:delta-simulates-A} and \eqref{Eq:semigroup-intersection-eq}:

\begin{claim*}
Suppose that in $G \ast F$ the equality 
\begin{equation}\label{Eq:free-product-expression-xyM}
(xy^{-1})^M = (\widetilde{p_1} v_1 \widetilde{q_1}^{-1}) (\widetilde{p_2} v_2 \widetilde{q_2}^{-1}) \cdots (\widetilde{p_n} v_n \widetilde{q_n}^{-1})
\end{equation}
holds for some $M \geq 1$, where $p_i \xrightarrow{v_i} q_i$ is a path in $\cA$ for all $1 \leq i \leq n$. Then we have $1 \in \pi(R)$. 
\end{claim*}
\begin{proof}[Proof of Claim.]
The proof will use a similar induction on $n$ as in the proof of \cite[Claim~1, proof of Lemma~11]{Lohrey2008}. If $n=1$, then we have $(xy^{-1})^M = \widetilde{p_n} v_n \widetilde{q_n}^{-1}$ in $G \ast F$, which by the normal form for free products clearly implies that $M=1$, which is the case dealt with in \cite{Lohrey2008}. 

Thus, assume $n>1$, and that the claim holds for $n-1$. There are several subcases to consider. First, suppose $q_i = p_{i+1}$ for some $i$. In this case, we can cancel these terms, and since we then have a path $p_i \xrightarrow{v_i v_{i+1}} q_{i+1}$, we can apply induction. On the other hand, if $q_i = p_i$ and $v_i = 1$ for some $1 \leq i \leq n$, then we can of course eliminate $\widetilde{p_i} v_i \widetilde{q_i}^{-1}$ from \eqref{Eq:free-product-expression-xyM} and obtain a shorter expression, and we are done by induction. Thus, the last case to consider is that $q_i \neq p_{i+1}$ and that $p_i = q_i$ implies $v_i \neq 1$ for all $1 \leq i \leq n$. Said otherwise, this says that the right-hand side of \eqref{Eq:free-product-expression-xyM} is in normal form, and hence, since the left-hand side is also in normal form, we must have $n= M$, $v_i = 1$ in $G$, as well as $\widetilde{p_i} = x$ and $\widetilde{q_i} = y$ for all $1 \leq i \leq n$. In particular, $p_1 = q_0$ and  $q_1 = q_f$, and since $p_1 \xrightarrow{v_1} q_1$, we have $q_0 \xrightarrow{1} q_f$. In other words, we have $1 \in L(\cA)$, as required. 
\end{proof}

Having proved the Claim, we now first establish that \eqref{Eq:delta-simulates-A} holds. Indeed, let us suppose first that $1 \in \pi(R)$, and let $w \in R$ be some word such that $\pi(w) = 1$. We can then write $w \equiv \sigma_1 \sigma_2 \cdots \sigma_n$ such that
\[
q_0 \xrightarrow{\sigma_1} q_1  \xrightarrow{\sigma_2} q_2 \xrightarrow{\sigma_3} \cdots \xrightarrow{\sigma_{n-1}} q_{n-1} \xrightarrow{\sigma_n} q_n
\]
is an accepting path in $\cA$, where $\sigma_i \in \oSigma \cup \{ \varepsilon \}$ for all $1 \leq i \leq n$. Thus we have 
\begin{align*}
\pi(xy^{-1}) &= \pi(xwy^{-1}) = \pi \left( \widetilde{q_0} ( \sigma_1 \sigma_2 \cdots \sigma_n) \widetilde{q_f} \right) \\
&= \pi\left( (\widetilde{q_0} \sigma_1  \widetilde{q_1}^{-1}) (\widetilde{q_1} \sigma_2  \widetilde{q_2}^{-1}) \cdots ( \widetilde{q}_{n-1} \sigma_n \widetilde{q_f}^{-1}) \right) \in \pi(\Delta^\ast),
\end{align*}
as desired. To prove the other direction, suppose that $\pi(xy^{-1}) \in \pi(\Delta^\ast)$. Then in $G \ast F$ we have the equality
\[
xy^{-1} = \widetilde{q_0} \widetilde{q_f}^{-1} =  (\widetilde{q_0} \sigma_1  \widetilde{q_1}^{-1})\ (\widetilde{q_1} \sigma_2  \widetilde{q_2}^{-1}) \cdots ( \widetilde{q}_{n-1} \sigma_n \widetilde{q_f}^{-1}) 
\]
where $\widetilde{q_i} \sigma_i  \widetilde{q_i} \in \Delta$ for all $1 \leq i \leq n$. Hence, by the Claim, we must have $1 \in \pi(R)$. This completes the proof of the equivalence \eqref{Eq:delta-simulates-A}, and hence also part (i) of the lemma.

Let us now prove the equivalence \eqref{Eq:semigroup-intersection-eq}, which will complete the proof of part (ii) of the lemma, and thereby also the entire lemma. If $\pi(xy^{-1}) \in \Mon(\Delta)$, then clearly as $xy^{-1} \neq 1$ in $G \ast F$, we have that $\pi(xy^{-1}) \in \Sgp(\Delta)$, and hence $\Sgp(xy^{-1}) \cap \Sgp(\Delta) \neq \varnothing$ as required. Suppose next, for the converse, that $\Sgp(\Delta) \cap \Sgp(xy^{-1}) \neq \varnothing$. Then some power of $xy^{-1}$ lies in $\Sgp(\Delta)$, i.e.\
\[
(xy^{-1})^M = (\widetilde{p_1} v_1 \widetilde{q_1}^{-1}) (\widetilde{p_2} v_2 \widetilde{q_2}^{-1}) \cdots (\widetilde{p_n} v_n \widetilde{q_n}^{-1})
\]
in $G \ast F$, where the $p_i \xrightarrow{a_i} q_i$ are paths in $\cA$. By the above claim, we have $1 \in \pi(L(\A))$. Hence, by \eqref{Eq:delta-simulates-A}, we have that $\pi(xy^{-1}) \in \Mon(\Delta)$. This completes the proof of \eqref{Eq:semigroup-intersection-eq}, and hence also the proof of the lemma. 
\end{proof}

\subsection{Further undecidability in $A(P_4)$}

In this section we shall state and prove several undecidability results for the right-angled Artin group $A(P_4)$ that are needed to prove the results about the braid group $\B_4$ in the next section. We fix some notation throughout this section. We fix the following presentation for $A(P_4)$: 
\begin{equation}
A(P_4) = \pres{a,b,c,d}{ab=ba, bc=cb, cd=dc} \label{A(P4)-presentation}
\end{equation}
In order to apply Lemma~\ref{Lem:reachability-for-RAT1-reduces-to-SMP-reach-xY} to the group $A(P_4)$ we first recall (Theorem~\ref{Thm:LS}) that the rational subset membership problem in $A(P_4)$, and indeed in any group containing a copy of $A(P_4)$, is undecidable. We now show undecidability of both the fixed-target submonoid membership problem, and the semigroup intersection problem, in the monoid $A(P_4)$; this proposition will then be used to yield our undecidability results in $\B_4$.   

To prove the proposition we shall need the following result which is a special case of results proved in \cite{Droms1991}. For the convenience of the reader we provide the details on how to deduce this lemma by applying results from \cite{Droms1991}. 

\begin{lemma}\label{lem:BiInfinite} Let
\[ G = A(P_4) = \pres{a,b,c,d}{ab=ba, bc=cb, cd=dc}. \]
For $i \geq 0$ set
\begin{align*}
  \alpha_i &= (ad)^i(aca^{-1})(ad)^{-i}, &   \beta_i &= (ad)^ib(ad)^{-i}, \\
  \mu_i &= (da)^ic(da)^{-i}, & \delta_i &= (da)^i(d b d^{-1})(da)^{-i}. 
\end{align*}
Then these elements are all distinct in $G$ and the subgroup $H = \langle X \rangle$ of $G$ generated by  
\[X = \{\alpha_i, \; \beta_i, \; \mu_i, \; \delta_i : i \geq 0 \}\]
is isomorphic to the RAAG $A(P_\infty)$ where $P_\infty$ is the bi-infinite line

\[
\begin{tikzpicture}[thick,scale=0.8]
\tikzstyle{lightnode}=[circle, draw, fill=black!20,
                        inner sep=0pt, minimum width=4pt]
\tikzstyle{darknode}=[circle, draw, fill=black!99,
                        inner sep=0pt, minimum width=4pt]
\draw[dotted] (-1,0) {} --(0,0);
\draw[dotted] (11,0) {} --(12,0);
\draw (0,0) node[lightnode] {} --(1,0) node[lightnode] {};
\draw (1,0) node[lightnode] {} --(2,0) node[lightnode] {};
\draw (2,0) node[lightnode] {} --(3,0) node[lightnode] {};
\draw (3,0) node[lightnode] {} --(4,0) node[lightnode] {};
\draw (4,0) node[lightnode] {} --(5,0) node[lightnode] {};
\draw (5,0) node[lightnode] {} --(6,0) node[lightnode] {};
\draw (6,0) node[lightnode] {} --(7,0) node[lightnode] {};
\draw (7,0) node[lightnode] {} --(8,0) node[lightnode] {};
\draw (8,0) node[lightnode] {} --(9,0) node[lightnode] {};
\draw (9,0) node[lightnode] {} --(10,0) node[lightnode] {};
\draw (10,0) node[lightnode] {} --(11,0) node[lightnode] {};
\draw (0,0) node[above] {$\alpha_2$};
\draw (1,0) node[above] {$\beta_2$};
\draw (2,0) node[above] {$\alpha_1$};
\draw (3,0) node[above] {$\beta_1$};
\draw (4,0) node[above] {$\alpha_0$};
\draw (5,0) node[above] {$\beta_0$};
\draw (6,0) node[above] {$\mu_0$};
\draw (7,0) node[above] {$\delta_0$};
\draw (8,0) node[above] {$\mu_1$};
\draw (9,0) node[above] {$\delta_1$};
\draw (10,0) node[above] {$\mu_2$};
\draw (11,0) node[above] {$\delta_2$};
\end{tikzpicture}
\]
\vspace{1mm}
\end{lemma}
\begin{proof} 
Since the words $\alpha_i$ and $\mu_i$ have $c$-exponent sum $1$ and the words $\beta_i$ and $\delta_i$ have $c$-exponent sum zero it follows that no word from $\{\alpha_i, \mu_i : i \geq 0 \}$ is equal in $G$ to any word from 
$\{\beta_i, \delta_i : i \geq 0 \}$. 
Note that all the words in the set 
$\{\alpha_i, \mu_i : i \geq 0 \}$ 
represent elements in the subgroup \[\langle a,d,c \rangle \cong 
\langle a \rangle \ast \langle c, d \rangle
\cong \mathbb{Z} \ast (\mathbb{Z} \times \mathbb{Z}) \]     
of $G$. Using  
normal forms in this free product
it then follows that  
distinct words from the set 
$\{\alpha_i, \mu_i : i \geq 0 \}$ 
represent distinct elements of $G$. 
A similar argument working in $\langle a,b,d\rangle \cong (\mathbb{Z} \times \mathbb{Z}) \ast \mathbb{Z}$ shows that  
distinct words from the set 
$\{\beta_i, \delta_i : i \geq 0 \}$
represent distinct elements of $G$. 
Hence the words in the set $X$ all represent distinct elements of the group $G$.   

Clearly $\alpha_0 = aca^{-1}$ commutes with $\beta_0 = b$, which upon conjugation by $(ad)^i$ implies that $\alpha_i$ commutes with $\beta_i$ for all $i$. Also $\alpha_i$ commutes with $\beta_{i+1}$ since $c$ commutes with $dbd^{-1}$. Conjugating these pairs by $d$ implies that $\mu_i$ commutes with $\delta_i$, and that $\mu_{i+1}$ commutes with $\delta_i$. Together with the fact that $\beta_0 = b$ commutes with $\mu_0=c$ this shows that words labelling   adjacent vertices in $P_{\infty}$ commute in $G$.  

Now let $N$ be the normal closure in $G$ of the subset  $\{b,c\} \subseteq G$. It follows from \cite[Theorem~2]{Droms1991} (see also the second paragraph of the proof of Theorem~3 in \cite{Droms1991}) that $N$ is isomorphic to a RAAG with generators being the set of distinct conjugates $C$ of $b$ and $c$ in $G$ by elements from the subgroup $\langle a,d \rangle \leq G$, and where the defining relators for the RAAG $N$ are given by commutators for each pair of elements from $C$ that commute with each other in $G$. It is also proved in \cite[Theorem~2 and Theorem~3]{Droms1991} that the underlying graph $\Phi$ of the RAAG $N$ is a forest. By definition all of the elements from $X$ belong to the set $C$. Hence, the subgroup $H$ of $G$ generated by $X$ is in fact the subgroup of $A(\Phi)$ generated by $X$. Since $X \subseteq \Phi$, it thus follows that $H$ is isomorphic to the RAAG on the subgraph of $\Phi$ induced by $X$, i.e.\ $H$ is a \textit{special} subgroup of $A(\Phi)$, see e.g.\ \cite[Lemma~3.2]{Vogtmann2015}. That is, $H$ is isomorphic to the RAAG with generating set $X$ and defining relations given by commutators for every pair of elements from $X$ that commute in $A(\Phi)$. Since $\Phi$ is a forest it follows that the only commutation relations between pairs of elements in the set $X$ are those corresponding to edges in the graph $P_\infty$, since any additional edge would create a cycle. This completes the proof that  $H \cong A(P_\infty)$.      
\end{proof}

\begin{proposition}\label{Prop:newUndec-for-P4}
Let $G=A(P_4)$ be the right-angled Artin group \eqref{A(P4)-presentation}. Then:
\begin{enumerate}
\item[(i)] The fixed-target submonoid membership problem for $aca^{-1}c^{-1} \in A(P_4)$ is undecidable. 
\item[(ii)] The semigroup intersection problem is undecidable; indeed there is no algorithm that takes a finite set $X \subseteq A(P_4)$ and decides whether or not $\mathrm{Sgp}( X ) \cap \mathrm{Sgp}( aca^{-1}c^{-1} ) = \varnothing$.
\end{enumerate}
\end{proposition}
\begin{proof} 
Let 
\[
G = A(P_4) = \pres{a,b,c,d}{ab=ba, bc=cb, cd=dc}.  
\]
Throughout the proof we use the same notation as in the statement of Lemma~\ref{lem:BiInfinite}.
%
Let $F = \langle \alpha_1, \alpha_0, \mu_0 \rangle \leq G$ and let  
$L = \langle \mu_1, \delta_1, \mu_2, \delta_2 \rangle \leq G$.  
It follows from Lemma~\ref{lem:BiInfinite} that $F$ is isomorphic to a free group of rank 3 with basis   
$\{ \alpha_1, \alpha_0, \mu_0  \}$.
It also follows from Lemma~\ref{lem:BiInfinite} that $L$ is isomorphic to the group $A(P_4)$, so
by Theorem~\ref{Thm:LS} the group $L$ has undecidable rational subset membership problem. 

It then follows from Lemma~\ref{Lem:reachability-for-RAT1-reduces-to-SMP-reach-xY}(i) that the fixed-target submonoid membership problem for $\alpha_0\mu_0^{-1}$ in $L \ast F$ is undecidable.
Indeed, if the fixed-target submonoid membership problem for $\alpha_0\mu_0^{-1}$ in $L \ast F$ were decidable then by 
Lemma~\ref{Lem:reachability-for-RAT1-reduces-to-SMP-reach-xY}(i) it would follow that the rational subset membership problem in $L$ is decidable, which it is not as $L \cong A(P_4)$.   

Now let 
\[
Y = \{\alpha_1, \alpha_0, \mu_0, \mu_1, \delta_1, \mu_2, \delta_2 \} \subseteq G 
\]
and let $K = \langle Y \rangle \leq G$ be the subgroup of $G$ generated by $Y$.  
Then by Lemma~\ref{lem:BiInfinite} we have 
\begin{eqnarray*}
  K & = & \langle Y \rangle = \langle \alpha_1, \alpha_0, \mu_0, \mu_1, \delta_1, \mu_2, \delta_2 \rangle \\
    & \cong & \langle  \alpha_1, \alpha_0, \mu_0 \rangle \ast \langle \mu_1, \delta_1, \mu_2, \delta_2 \rangle \cong  
L \ast F.
\end{eqnarray*}
Since by the previous paragraph the fixed-target submonoid membership problem for $\alpha_0\mu_0^{-1}$ in $L \ast F$ is undecidable, and $K \cong L \ast F$ vie the isomorphism described above it follows that the fixed-target submonoid membership problem for $\alpha_0\mu_0^{-1}$ in $K$ is undecidable.    

Note that by definition $\alpha_0 = aca^{-1}$ and $\mu_0 = c$ hence $\alpha_0\mu_0^{-1} = aca^{-1}c^{-1}$ in $G$.      
Since the fixed-target submonoid membership problem for $\alpha_0\mu_0^{-1}$ in $K$ is undecidable,
and since $K$ is a finitely generated subgroup of $G$,   
it then follows by the contrapositive of Remark~\ref{rmk:preservation} that the fixed-target submonoid membership problem for
$\alpha_0\mu_0^{-1}$ in $G$ is undecidable. 
But since $\alpha_0\mu_0^{-1} = aca^{-1}c^{-1}$ in $G$ this means that the fixed-target submonoid membership problem for
$aca^{-1}c^{-1}$ in $G$ is undecidable. This completes the proof of part (i) of the proposition.  

The proof of part (ii) of the proposition is dealt with in a similar way but by applying Lemma~\ref{Lem:reachability-for-RAT1-reduces-to-SMP-reach-xY}(ii). The details are as follows. 

Maintaining the same notation as earlier in the current proof, 
it follows from Lemma~\ref{Lem:reachability-for-RAT1-reduces-to-SMP-reach-xY}(ii) that 
the semigroup intersection problem in $L \ast F$ with respect to the subsemigroup $\operatorname{Sgp}(\alpha_0\mu_0^{-1})$ is undecidable. Indeed, if the 
semigroup intersection problem were decidable in $L \ast F$ with respect to the subsemigroup $\operatorname{Sgp}(\alpha_0\mu_0^{-1})$ then by Lemma~\ref{Lem:reachability-for-RAT1-reduces-to-SMP-reach-xY}(ii) it would follow that the rational subset membership problem in $L$ is decidable, which it is not as $L \cong A(P_4)$.   
Then since 
\begin{eqnarray*}
  K = \langle Y \rangle \cong \langle  \alpha_1, \alpha_0, \mu_0 \rangle \ast \langle \mu_1, \delta_1, \mu_2, \delta_2 \rangle \cong  
L \ast F
\end{eqnarray*}
it follows that the semigroup intersection problem in $K$ with respect to the subsemigroup $\operatorname{Sgp}(\alpha_0\mu_0^{-1})$ is undecidable.

Since the semigroup intersection problem in $K$ with respect to the subsemigroup $\operatorname{Sgp}(\alpha_0\mu_0^{-1})$ is undecidable and since $K$ is a finitely generated subgroup of $G$, it then follows by the contrapositive of Remark~\ref{rmk:preservation} that the semigroup intersection problem in $G$ with respect to the subsemigroup $\operatorname{Sgp}(\alpha_0\mu_0^{-1})$ is undecidable. 
In more detail suppose, seeking a contradiction, that the semigroup intersection problem in $G$ with respect to the subsemigroup $\operatorname{Sgp}(\alpha_0\mu_0^{-1})$ \emph{is} decidable. This would mean that there is an algorithm that for any finite set of words $W \subseteq \{a,b,c,d, a^{-1}, b^{-1}, c^{-1}, d^{-1} \}^*$ can decide whether or not
\[
\operatorname{Sgp}(\alpha_0\mu_0^{-1}) \cap \operatorname{Sgp}(W) = \varnothing.  
\]    
But then in particular given any finite set of words $U \subseteq (Y \cup Y^{-1})^*$ 
we can rewrite these words via the substitutions  
\begin{align*}
  \alpha_i &\mapsto (ad)^i(aca^{-1})(ad)^{-i}, &   \beta_i & \mapsto (ad)^ib(ad)^{-i}, \\
  \mu_i & \mapsto (da)^ic(da)^{-i}, & \delta_i & \mapsto (da)^i(d b d^{-1})(da)^{-i} 
\end{align*}
to a finite set of words $U' \subseteq \{a,b,c,d, a^{-1}, b^{-1}, c^{-1}, d^{-1} \}^*$ representing the same set of elements of $G$ as $U$. Then by assumption we can decide whether or not 
\[
\operatorname{Sgp}(\alpha_0\mu_0^{-1}) \cap \operatorname{Sgp}(U') = \varnothing \ \mbox{in $G$.}  
\] 
But since $\operatorname{Sgp}(U') \leq K$ and  
$\operatorname{Sgp}(\alpha_0\mu_0^{-1}) \leq K$ it follows that 
\[
\operatorname{Sgp}(\alpha_0\mu_0^{-1}) \cap \operatorname{Sgp}(U') = \varnothing \ \mbox{in $G$} 
\]
if and only if  
\[
\operatorname{Sgp}(\alpha_0\mu_0^{-1}) \cap \operatorname{Sgp}(U) = \varnothing \ \mbox{in $K$} 
\]
where $U$ is the original finite subset $U \subseteq (Y \cup Y^{-1})^*$ of words over the generators of $K$. 
This shows that there is an algrothm that takes any finite subset $U \subseteq (Y \cup Y^{-1})^*$ and decides whether or not 
$\operatorname{Sgp}(\alpha_0\mu_0^{-1}) \cap \operatorname{Sgp}(U) = \varnothing$
in $K = \langle Y \rangle$. 
This gives the required contradiction, since the semigroup intersection problem in $K$ with respect to the subsemigroup $\operatorname{Sgp}(\alpha_0\mu_0^{-1})$ is undecidable.  

%
%
%

Finally, since $\alpha_0\mu_0^{-1} = aca^{-1}c^{-1}$ in $G$ this means that the 
semigroup intersection problem in $G$ with respect to the subsemigroup $\operatorname{Sgp}(aca^{-1}c^{-1})$ is undecidable.
This completes the proof of part (ii) of the proposition.  
\end{proof}

\subsection{Undecidability in the braid group $\B_4$}\label{Sec:undecidability-in-b4}

In this section we state and prove several undecidability results about the braid group $\B_4$. Each of the results will follow from the corresponding result for the group $A(P_4)$ proved in \S\ref{Sec:AP4-undec} together with  
 the fact (Theorem~\ref{Thm:Droms-Ap4-embeds}) that $A(P_4)$ embeds as a subgroup of $\B_4$ and the fact (see Remark~\ref{rmk:preservation}) that decidability of each of these problems is inherited by finitely generated subgroups of groups.  

\begin{theorem}\label{Thm:Undecidability-b4}
Let $\B_4$ be the braid group on four strands. Then in $\B_4$:
\begin{enumerate}
\item[(i)] The submonoid membership problem is undecidable; indeed, there is a fixed finite subset $X \subseteq \B_4$ such that there is no algorithm deciding membership in $\operatorname{Mon}(X) \leq \B_4$. 
\item[(ii)] The fixed-target submonoid membership problem for $\gamma_0 \in \B_4$ is undecidable, where $\gamma_0$ is the braid in Figure~\ref{Fig:Undecidable-braid}, that is, there is no algorithm that takes a finite set of braids $X \subseteq \B_4$ and decides whether or not $\gamma_0 \in \mathrm{Mon}( X )$. 
\item[(iii)] The subsemigroup intersection problem is undecidable; indeed, there is no algorithm that takes a finite set of braids $X \subseteq \B_4$ and decides whether or not $\mathrm{Sgp}( X ) \cap \mathrm{Sgp}( \gamma_0 ) = \varnothing$. \end{enumerate}
\begin{figure}[h]
\begin{tikzpicture}
\braid[rotate=90, number of strands=4,
       style strands={5,6}{draw=none}] (braid) a_2 a_2 a_1 a_1 a_2^{-1} a_2^{-1} a_1^{-1} a_1^{-1};
\end{tikzpicture}
\caption{The braid $\gamma_0$ from Theorem~\ref{Thm:Undecidability-b4}.}
\label{Fig:Undecidable-braid}
\end{figure}
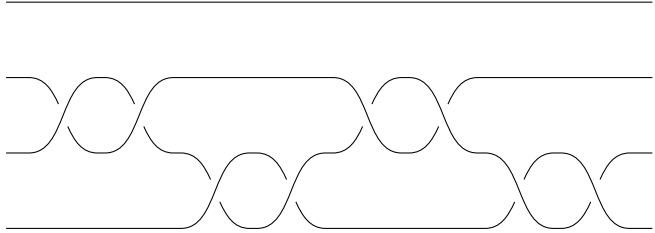
\end{theorem}

\begin{proof}
We prove each statement in turn.

\begin{enumerate}
\item[(i)] In \cite{Lohrey2008} it is proved that there is a fixed finite subset $X \subseteq A(P_4)$ such that there is no algorithm deciding membership in $\Mon(X) \leq A(P_4)$. Since $A(P_4) \leq \B_4$ by Theorem~\ref{Thm:Droms-Ap4-embeds}, the same undecidability result is also true for $\B_4$. 
\item[(ii)] By Theorem~\ref{Thm:Droms-Ap4-embeds}, $A(P_4)$ is isomorphic to the subgroup of $\B_4$ generated by $\sigma_2^2, (\sigma_2 \sigma_3 \sigma_2)^2, \sigma_3^2, \sigma_1^2$, with the generators given in the same order as in $P_4$ and where the $\sigma_i$ are the generators of $\B_4$ in \eqref{Eq:braid-group-presentation}. Thus, by Proposition~\ref{Prop:newUndec-for-P4}(i), taking $a = \sigma_2^2$ and $c = \sigma_3^2$, we get that the fixed-target submonoid membership problem is undecidable for $\gamma_0 = \sigma_2^2 \sigma_3^2 \sigma_2^{-2} \sigma_3^{-2}$, which is precisely the element in Figure~\ref{Fig:Undecidable-braid}. 
\item[(iii)] This follows from using the same element $\gamma_0$ as in (ii) and applying Proposition~\ref{Prop:newUndec-for-P4}(ii).\end{enumerate}\end{proof}

Recall from \S\ref{Subsec:decision problems} that the fixed-target subsemigroup membership problem for $1 \in \B_4$ is usually called the \textit{identity problem}. We do not know whether this problem is decidable in $\B_4$. However, Theorem~\ref{Thm:Undecidability-b4}(ii) shows that if we consider the fixed-target subsemigroup membership problem for $\gamma_0$ rather than that for $1$, then the problem becomes undecidable. This might be regarded as evidence that identity problem may also be undecidable in $\B_4$. Other problems that remain open for $\B_4$ are the group problem and the subgroup membership problem. Note that $\B_4$ is known to be incoherent \cite{Gordon2004}, i.e.\ it contains finitely generated subgroups which are not finitely presented. This shows that the subgroup structure of $\B_4$ is already complicated, and lends credence to the idea that the subgroup membership problem may be undecidable in $\B_4$. 

We end this section by remarking that all the decidability results proved in this section for braid groups also hold true for a closely related object called the \emph{pure braid group} $\PB_n$. 
Indeed, it well-known and also clear from equation \eqref{Eq:braid-group-presentation} above (and the graphical representation of braids see 
\cite[\S 1.2.2]{kassel2008braid}) that $\B_n$ surjects onto $S_n$. The kernel of this homomorphism is called the \textit{pure braid group} $\PB_n$, which hence has index $n!$ in $\B_n$. All of the undecidability results above for  $\B_n$, were proved by means of embedding certain right-angled Artin groups $A(\Gamma)$ into $\B_n$. However, the results are equally valid for $\PB_n$, by the following reasoning. 
It follows from the main result of \cite{Crisp2001} that if $A(\Gamma)$ is a right-angled Artin group  with generating set $V$ (defined by the presentation in equation \eqref{Eq:raag-presentation}) then for any natural number $k \in \mathbb{N}$ with $k \geq 2$ the subgroup of $A(\Gamma)$ generated by the set of powers $\{ v^k : v \in V \}$ is again isomorphic to the group $A(\Gamma)$.        
Hence by taking powers which are sufficiently large multiples of $[\B_n : \PB_n] = n!$ it is not difficult to see that the following statement holds: for every finite graph $\Gamma$, 
there is a subgroup of $\B_n$ that is isomorphic to $A(\Gamma)$ if and only if there is a subgroup of $\PB_n$ that is isomorphic to $A(\Gamma)$.   
Thus, all undecidability results above for $\B_n$ with $n \geq 4$ also hold for $\PB_n$ with $n \geq 4$. To see that all the statements in Theorem~\ref{Thm:Undecidability-b4} remain true for the pure braid group $\PB_4$, it simply suffices to note that $\gamma_0 \in \PB_4$. Similarly, when $n \leq 3$, since $\B_n$ has decidable rational subset membership problem and $\PB_n \leq \B_n$ it follows that the rational subset membership problem (along with all other decision problems discussion in this section) is decidable in $\PB_n$ when $n \leq 3$.     

\section{Artin groups}\label{Sec:ArtinGroups}

In Theorem~\ref{Thm:Undecidability-b4}
above we give a list of undecidability results for $B_4$ obtained by applying the general results for $A(P_4)$ established earlier in the paper. That result was the original motivation for the work done in this article, which for the decision problems listed in that theorem shows that the problem is decidable in $\B_n$ if and only if $n \leq 3$. Somewhat surprisingly it turns out that, using similar methods, for each of the decision problems given in Theorem~\ref{Thm:Undecidability-b4} we can classify exactly for which Artin groups the given problem is decidable\footnote{The authors wish to thank Sang-hyun Kim (KIAS) for suggesting that our results and methods above for $\B_4$ might also be applied successfully to other Artin groups, and for other helpful discussions pertaining to right-angled Artin groups and braid groups.}. This classification will be given in terms of the existence of certain subgraphs in the defining graph. This result provides a common generalization of our characterization of braid groups with decidable submonoid membership problem above, and the theorem of Lohrey \& Steinberg from \cite{Lohrey2008} which characterizes the right-angled Artin groups with decidable submonoid membership problem. As well as developing and applying some of the ideas from \cite{Lohrey2008}, the other key ingredient for us is the recent work of Almeida \& Lima \cite{AlmeidaLima2021, AlmeidaLima2024} who classified the subgroup separable Artin groups. It will turn out that the conditions needed for an Artin group to be subgroup separable exactly match the conditions needed for it to have decidable submonoid (and rational subset) membership problem. Recall that a group $G$ is \emph{subgroup separable} if for every finitely generated subgroup $H$ of $G$, and element $g \in G \setminus H$ there is a finite index subgroup $K$ of $G$, which contains $H$ but $g \not\in K$.       

Following \cite{AlmeidaLima2021} we use $A(\mathcal{S})$ to denote the smallest class of Artin groups that contain all Artin groups of rank at most two, and has the following properties 
\begin{itemize} 
\item $A, B \in A(\mathcal{S}) \Rightarrow A \ast B \in A(\mathcal{S}) $; 
\item $A \in A(\mathcal{S}) \Rightarrow A \times \mathbb{Z} \in A(\mathcal{S}) $.  
  \end{itemize}
It follows from recent work \cite[Corollary A]{AlmeidaLima2021} that every subgroup separable Artin group lies in $A(\mathcal{S})$. The first result we will establish is the following. 

\begin{theorem}\label{thm_subgroup_sep_Artin_groups} 
Every group in $A(\mathcal{S})$ has decidable rational subset membership problem. In particular, every subgroup separable Artin group has decidable rational subset membership problem.  
\end{theorem}

In \cite[Theorem 3.1]{Kambites2007} it is shown that a finitely generated group $G$ has decidable rational subset membership problem if and only if the word problem for $G$ with respect to $A$ is RID (rational intersection decidable). Here the word problem for the group is the set of all words over the finite generating set that are equal to the identity element of the group. In \cite{Lohrey2008} Lohrey \& Steinberg introduce another class called SLI (semilinear intersection) languages, and corresponding class of SLI-groups. They prove \cite[Lemma~4]{Lohrey2008} that this class is  
contained in the class of languages RID, so every SLI-group is an RID-group and hence has decidable rational subset membership problem. 

We will prove Theorem~\ref{thm_subgroup_sep_Artin_groups} by showing that every subgroup separable Artin group is an SLI-group. This can be done by first verifying that all rank 1 and rank 2 Artin groups are SLI-groups, and then applying some closure properties from of SLI-groups from \cite{Lohrey2008}.  

We will not need to define SLI-groups here but just list some examples and basic properties in the following result. For a definition of SLI-group we refer the reader to \cite[Section~3]{Lohrey2008}. 

\begin{lemma}\label{lem_LS_closureProperties} Let $G$ and $H$ be finitely generated groups.    
\begin{enumerate}
\item[(i)] 
If $G$ is an SLI-group, then $G \times \Z$ is also an SLI-group. 
\item[(ii)] 
If $G$ and $H$ are SLI-groups, then $G \ast H$ is also an SLI-group. 
\item[(iii)] 
If $H$ is a finite index subgroup of $G$, and $H$ is an SLI-group, then $G$ is an SLI-group.      
\item[(iv)] 
Every finitely generated free group is an SLI-group.  
  \end{enumerate}
\end{lemma}
\begin{proof} 
Parts (i), (ii) and (iii) are Lemmas~6, 7 and 3, respectively, in \cite{Lohrey2008}. Part (iv) follows from (i) and (ii), together with the fact that the trivial group has decidable rational subset membership problem so is an RID group and thus an SLI-group; alternatively, one can use the fact that finite rank free groups have decidable rational subset membership problem. 
  \end{proof}

Recall that an Artin group of rank two is known in the literature as a \textit{dihedral} Artin group. It is also well known that dihedral Artin groups admit particularly simple presentations (see e.g. 
\cite[Proof of Theorem 5.1]{CrispParis2005}), namely one of the presentations:
\begin{align}
\pres{a, b}{(a b)^k a = b(a b)^k} & \cong \pres{x, y}{x^2 = y^{2k+1}}, \mbox{and} \label{Eq:torus-knot-case}\\
\pres{a, b}{(a b)^k = (b a)^k} & \cong \pres{x, t}{x^ky = yx^k} \cong \operatorname{BS}(k,k) \label{Eq:BS(k,k)-case}
\end{align}
where $k \geq 1$, and $\operatorname{BS}(k, k)$ denotes a (unimodular) Baumslag--Solitar group.

\begin{lemma}\label{lem_rank2Artin}
Let $A$ be an Artin group with rank at most $2$. Then $A$ is an SLI-group.    
\end{lemma}
\begin{proof} 
An Artin group of rank 1 is $\mathbb{Z}$ and so is an SLI-group by Lemma~\ref{lem_LS_closureProperties}.  Now suppose that $A$ is an Artin group of rank 2. Then $A$ is either isomorphic to the group in \eqref{Eq:torus-knot-case}, or \eqref{Eq:BS(k,k)-case}. In the first case \eqref{Eq:torus-knot-case}, $A$ is isomorphic to a torus knot group, and it is well-known (see e.g. \cite[Theorem~2.1]{Niblo2001}) that every such group is virtually $F_l \times \mathbb{Z}$. Hence by Lemma~\ref{lem_LS_closureProperties} it follows that in this case $A$ is an SLI-group. In the second case when $A$ is given by \eqref{Eq:BS(k,k)-case}, we have $A \cong \operatorname{BS}(k,k)$, which is also well-known to contain a finite index subgroup isomorphic to $F_l \times \mathbb{Z}$; see e.g.\ \cite{Gilbert2015}.  
Hence, again by applying 
Lemma~\ref{lem_LS_closureProperties}
it follows that in this case $A$ is also an SLI-group. 
\end{proof}

\begin{proof}[Proof of Theorem~\ref{thm_subgroup_sep_Artin_groups}] 
By Lemma~\ref{lem_rank2Artin} every Artin group with rank at most $2$ is an SLI-group. Together with Lemma~\ref{lem_LS_closureProperties}, this yields that every group in $A(\mathcal{S})$ is an SLI-group. Hence by \cite[Lemma 4]{Lohrey2008} every group in $A(\mathcal{S})$ has decidable rational subset membership problem. By \cite[Corollary A]{AlmeidaLima2021} every subgroup separable Artin group is in $A(\mathcal{S})$, completing the proof.
  \end{proof}

The following theorem classifies the Artin groups with decidable submonoid membership problem. Note that the classification of braid groups with decidable submonoid membership problem we obtained above, and the classification of right-angled Artin groups with decidable submonoid membership problem proved in \cite{Lohrey2008},
can both be recovered as special cases of this general theorem for Artin groups. 

\begin{theorem}\label{thm_Artin_in_general} 
Let $A = A(\Gamma)$ be an Artin group. 
Then the following are equivalent: 
\begin{enumerate} 
\item[(i)] $A$ has decidable submonoid membership problem;
\item[(ii)] $A$ has decidable rational subset membership problem;
\item[(iii)] the graph $\Gamma$ does not embed any of the following graphs as an induced subgraph: 
  \begin{enumerate} 
  \item a (generalized) square of one of the following three forms   
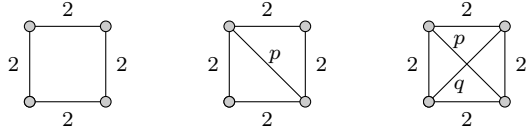
\[
\begin{tikzpicture}
\tikzstyle{lightnode}=[circle, draw, fill=black!20,
                        inner sep=0pt, minimum width=4pt]
\tikzstyle{darknode}=[circle, draw, fill=black!99,
                        inner sep=0pt, minimum width=4pt]
\draw (0,0) node[lightnode] {} --(1,0);
\draw (1,0) node[lightnode] {} --(1,1);
\draw (1,1) node[lightnode] {} --(0,1);
\draw (0,1) node[lightnode] {} --(0,0) node[lightnode] {};
\draw (0.5,0) node[below] {\scriptsize $2$};
\draw (0.5,1) node[above] {\scriptsize $2$};
\draw (0,0.5) node[left] {\scriptsize $2$};
\draw (1,0.5) node[right] {\scriptsize $2$};
\end{tikzpicture}
\quad
\quad
\begin{tikzpicture}
\tikzstyle{lightnode}=[circle, draw, fill=black!20,
                        inner sep=0pt, minimum width=4pt]
\tikzstyle{darknode}=[circle, draw, fill=black!99,
                        inner sep=0pt, minimum width=4pt]
\draw (1,0) --(0,1);
\draw (0,0) node[lightnode] {} --(1,0);
\draw (1,0) node[lightnode] {} --(1,1);
\draw (1,1) node[lightnode] {} --(0,1);
\draw (0,1) node[lightnode] {} --(0,0) node[lightnode] {};
\draw (0.5,0) node[below] {\scriptsize $2$};
\draw (0.5,1) node[above] {\scriptsize $2$};
\draw (0,0.5) node[left] {\scriptsize $2$};
\draw (1,0.5) node[right] {\scriptsize $2$};
\draw (0.6,0.6) node {\scriptsize $p$};
\end{tikzpicture}
\quad
\quad
\begin{tikzpicture}
\tikzstyle{lightnode}=[circle, draw, fill=black!20,
                        inner sep=0pt, minimum width=4pt]
\tikzstyle{darknode}=[circle, draw, fill=black!99,
                        inner sep=0pt, minimum width=4pt]
\draw (1,0) --(0,1);
\draw (0,0) --(1,1);
\draw (0,0) node[lightnode] {} --(1,0);
\draw (1,0) node[lightnode] {} --(1,1);
\draw (1,1) node[lightnode] {} --(0,1);
\draw (0,1) node[lightnode] {} --(0,0) node[lightnode] {};
\draw (0.5,0) node[below] {\scriptsize $2$};
\draw (0.5,1) node[above] {\scriptsize $2$};
\draw (0,0.5) node[left] {\scriptsize $2$};
\draw (1,0.5) node[right] {\scriptsize $2$};
\draw (0.4,0.8) node {\scriptsize $p$};
\draw (0.4,0.2) node {\scriptsize $q$};
\end{tikzpicture}
\]
with $p >2$ and $q>2$, 
\item a triangle of the form
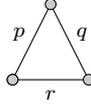
\[ 
\begin{tikzpicture}
\tikzstyle{lightnode}=[circle, draw, fill=black!20,
                        inner sep=0pt, minimum width=4pt]
\tikzstyle{darknode}=[circle, draw, fill=black!99,
                        inner sep=0pt, minimum width=4pt]
\draw (0,0) node[lightnode] {} --(1,0);
\draw (1,0) node[lightnode] {} --(0.5,1);
\draw (0.5,1) node[lightnode] {} --(0,0) node[lightnode] {};
\draw (0.5,0) node[below] {\scriptsize $r$};
\draw (0.3,0.6) node[left] {\scriptsize $p$};
\draw (0.7,0.6) node[right] {\scriptsize $q$};
\end{tikzpicture}
\]
where at most one of $\{p,q,r\}$ is equal to $2$, or
\item a path of one of the following two forms  
\[ 
\begin{tikzpicture}
\tikzstyle{lightnode}=[circle, draw, fill=black!20,
                        inner sep=0pt, minimum width=4pt]
\tikzstyle{darknode}=[circle, draw, fill=black!99,
                        inner sep=0pt, minimum width=4pt]
\draw (0,0) node[lightnode] {} --(1,0);
\draw (1,0) node[lightnode] {} --(2,0);
\draw (2,0) node[lightnode] {} --(3,0) node[lightnode] {};
\draw (0.5,0) node[below] {\scriptsize $2$};
\draw (1.5,0) node[below] {\scriptsize $2$};
\draw (2.5,0) node[below] {\scriptsize $2$};
\end{tikzpicture}
\quad
\quad
\begin{tikzpicture}
\tikzstyle{lightnode}=[circle, draw, fill=black!20,
                        inner sep=0pt, minimum width=4pt]
\tikzstyle{darknode}=[circle, draw, fill=black!99,
                        inner sep=0pt, minimum width=4pt]
\draw (0,0) node[lightnode] {} --(1,0);
\draw (1,0) node[lightnode] {} --(2,0) node[lightnode] {};
\draw (0.5,0) node[below] {\scriptsize $p$};
\draw (1.5,0) node[below] {\scriptsize $q$};
\end{tikzpicture}
\]
where at most one of $\{p,q\}$ is equal to $2$.   
  \end{enumerate}    
\end{enumerate}
Moreover, if $\Gamma$ does embed one of the graphs in (a), (b) or (c) as an induced subgraph then $A(\Gamma)$ contains a fixed finitely generated submonoid in which membership is undecidable. 
\end{theorem}
\begin{proof} 
((iii) $\Rightarrow$ (ii)) 
By \cite[Theorem 1.1]{AlmeidaLima2024} (which is proved using the main result of \cite{AlmeidaLima2021}) it follows that 
if $\Gamma$ does not embed any of the graphs in (a), (b), or (c) then  
$A(\Gamma)$ is subgroup separable which by Theorem~\ref{thm_subgroup_sep_Artin_groups} implies that $A$ has decidable rational subset membership problem. 
((ii) $\Rightarrow$ (i)) is trivial.  
((i) $\Rightarrow$ (iii)) 
Suppose that $\Gamma$ does embed one of the graphs of (a), (b) or (c) as an induced subgraph. 
Let us consider each case in turn.   
Case (a): In this case by \cite{Crisp2001} the RAAG $A(C_4) \cong F_2 \times F_2$ embeds into $A(\Gamma)$, which hence contains a fixed finitely generated submonoid (in fact subgroup) in which membership is undecidable by \cite{Mikhailova1958}.    
Cases (b) and (c): In these cases it follows from the argument of the proof of \cite[Lemma 4.2]{AlmeidaLima2021} that in each case $A(\Gamma)$ contains\footnote{We note that diagram (4.1) in the proof of \cite[Lemma~4.2]{AlmeidaLima2021} there is an edge between $x$ and $y$ that should not be there. Indeed, by the definition of $R(B)$ in that paper for there to be an edge between $x$ and $y$ one would need every vertex of $\{a,b\}$ to be connected to every vertex of $\{b,c\}$ by a $2$-labelled edge in the defining graph of the Artin group which is not true. This does not affect the validity of the proof of \cite[Lemma~4.2]{AlmeidaLima2021}, or the proof of Theorem~\ref{thm_Artin_in_general} above, since the RAAG of the graph (4.1) with that edge removed still contains an induced subgraph isomorphic to a path of length $4$. Related to this, note that since the braid group $B_4 \cong A(2,3,3)$ does not embed $A(C_4) = F_2 \times F_2$ it follows that (4.1) cannot contain a square as an induced subgraph.} the RAAG $A(P_4)$ and hence by \cite{Lohrey2008} it follows that $A(\Gamma)$ contains a fixed finitely generated submonoid in which membership is undecidable.       
  \end{proof}
In the same way as for braid groups in Theorem~\ref{Thm:Undecidability-b4} above, for Artin groups with undecidable submonoid membership problem we can apply our results to identify other algorithmic problems that are also undecidable, obtaining the following.    
\begin{theorem} \label{thm_Artin_in_general_OtherAlgProblems}
Let $A = A(\Gamma)$ be an Artin group with undecidable submonoid membership problem, i.e.\ such that one of the graphs in (a), (b) or (c) from the statement of Theorem~\ref{thm_Artin_in_general} embeds into $\Gamma$.
\begin{enumerate} 
\item[(i)] If one of the graphs in (b) or (c) embeds into $\Gamma$, then $A(P_4)$ embeds into $A(\Gamma)$, and the following problems are all undecidable in $A(\Gamma)$: 
the rational subset membership problem, 
submonoid membership problem,
fixed-target submonoid membership problem,
and the subsemigroup intersection problem.
\item[(ii)]  If one of the graphs in (a) embeds into $\Gamma$, then $A(C_4)$ embeds into $A(\Gamma)$, and all the problems listed in (i) are undecidable in $A(\Gamma)$. Additionally, all the following problems are undecidable in $A(\Gamma)$: the group problem, identity problem, and subgroup membership problem.    
  \end{enumerate} 
\end{theorem}

\begin{corollary} \label{corol_Artin_in_general_OtherAlgProblems}
Each and every one of the following problems is decidable for an Artin group $A = A(\Gamma)$ 
if and only if none of the graphs in (a), (b) and (c) from the statement of Theorem~\ref{thm_Artin_in_general} embeds into $\Gamma$:
the rational subset membership problem, 
submonoid membership problem,
fixed-target submonoid membership problem,
and the subsemigroup intersection problem.
\end{corollary}

Just as for the braid group $\B_4$ we do not know exactly which Artin groups have decidable identity problem or group problem. Similarly, it is unknown exactly which Artin groups have decidable subgroup membership problem. In particular, it remains open whether the braid group $\B_4$ has decidable subgroup membership problem. Related to this, it follows from the proof of \cite[Lemma~4.2]{AlmeidaLima2021} that $\B_4$ contains $A(C_5)$, where $C_5$ is a pentagon, and it is also an open problem whether $A(C_5)$ has decidable subgroup membership problem; see \cite[Section~5]{Lohrey2008}.  

A natural question arising from the results in this section is whether one might also be able to classify the Coxeter groups with decidable submonoid membership problem. One current obstacle to this is that there are many Coxeter groups (even right-angled Coxeter groups) that are virtually hyperbolic surface groups 
(see e.g. \cite[Theorem~2.1]{Gordon2004}), and
the rational and submonoid membership problems are both currently open for hyperbolic surface groups.  
In particular, it follows from \cite[Theorem~2.1]{Gordon2004} that the right-angled Coxeter group of the pentagon with presentation 
\[
\pres{a_0,a_1,a_2,a_3,a_4}{a_i^2=1, a_i a_{i+1} = a_{i+1}a_i \; (i \in \{0,1,2,3,4\})} 
\] 
with subscripts taken modulo $5$,  
contains a hyperbolic surface subgroup of finite index, so it would be interesting to determine whether or not this group has decidable submonoid membership problem or rational subset membership problem. 

\bibliography{BobCFArtinJLMS-FinalVersion-30Sep2025} 

\newcommand{\etalchar}[1]{$^{#1}$}
\providecommand{\bysame}{\leavevmode\hbox to3em{\hrulefill}\thinspace}
\providecommand{\MR}{\relax\ifhmode\unskip\space\fi MR }
\providecommand{\MRhref}[2]{%
  \href{http://www.ams.org/mathscinet-getitem?mr=#1}{#2}
}
\providecommand{\href}[2]{#2}
\begin{thebibliography}{BGCMW22}

\bibitem[Aki91]{Akimenkov1991}
A.~M. Akimenkov, \emph{Subgroups of the braid group {$B_4$}}, Mat. Zametki
  \textbf{50} (1991), no.~6, 3--13, 156.

\bibitem[AL21]{AlmeidaLima2021}
Kisnney Almeida and Igor Lima, \emph{Subgroup separability of {A}rtin groups},
  Journal of Algebra \textbf{583} (2021), 25--37.

\bibitem[AL24]{AlmeidaLima2024}
\bysame, \emph{Subgroup separability of {A}rtin groups {II}}, Pre-print (2024),
  Available online at \texttt{arXiv:2403.05483}.

\bibitem[{A}rt47]{artin1947theory}
Emil {A}rtin, \emph{Theory of braids}, Annals of Mathematics \textbf{48}
  (1947), no.~1, 101--126.

\bibitem[BCMWR24]{boyd2024artin}
Rachael Boyd, Ruth Charney, Rose Morris-Wright, and Sarah Rees, \emph{The
  {A}rtin monoid cayley graph}, Journal of Combinatorial Algebra (2024).

\bibitem[BGCH{\etalchar{+}}24]{BlascoGarciaCumplidoHoltMorrisWrightRees2024}
Rub\'en Blasco-Garc\'ia, Mar\'ia Cumplido, Derek~F. Holt, Rose Morris-Wright,
  and Sarah Rees, \emph{Rewriting in {A}rtin groups without {$A_3$} or {$B_3$}
  subdiagrams}, arXiv preprint arXiv:2412.12195 (2024).

\bibitem[BGCMW22]{blasco2022word}
Rub{\'e}n Blasco-Garc{\'\i}a, Mar{\'\i}a Cumplido, and Rose Morris-Wright,
  \emph{The word problem is solvable for 3-free artin groups in quadratic
  time}, arXiv preprint arXiv:2204.03523 (2022).

\bibitem[BGGM07a]{BirmanGebhardt2007}
Joan~S. Birman, Volker Gebhardt, and Juan Gonz\'{a}lez-Meneses, \emph{Conjugacy
  in {G}arside groups. {I}. {C}yclings, powers and rigidity}, Groups Geom. Dyn.
  \textbf{1} (2007), no.~3, 221--279. \MR{2314045}

\bibitem[BGGM07b]{BirmanGebhardt2008}
\bysame, \emph{Conjugacy in {G}arside groups. {III}. {P}eriodic braids}, J.
  Algebra \textbf{316} (2007), no.~2, 746--776. \MR{2358613}

\bibitem[BGGM08]{BirmanGebhardt2007b}
\bysame, \emph{Conjugacy in {G}arside groups. {II}. {S}tructure of the ultra
  summit set}, Groups Geom. Dyn. \textbf{2} (2008), no.~1, 13--61. \MR{2367207}

\bibitem[BGP22]{BlascoGarcia2022}
Rub\'{e}n Blasco-Garc\'{\i}a and Luis Paris, \emph{On the isomorphism problem
  for even {A}rtin groups}, J. Algebra \textbf{607} (2022), no.~part B, 35--52.
  \MR{4441309}

\bibitem[Bir74]{Birman1974}
Joan~S. Birman, \emph{Braids, links, and mapping class groups}, Annals of
  Mathematics Studies, No. 82, Princeton University Press, Princeton, NJ;
  University of Tokyo Press, Tokyo, 1974.

\bibitem[BKL98a]{birman1998new}
Joan Birman, Ki~Hyoung Ko, and Sang~Jin Lee, \emph{A new approach to the word
  and conjugacy problems in the braid groups}, Advances in Mathematics
  \textbf{139} (1998), no.~2, 322--353.

\bibitem[BKL98b]{BirmanKoLee1998}
\bysame, \emph{A new approach to the word and conjugacy problems in the braid
  groups}, Adv. Math. \textbf{139} (1998), no.~2, 322--353. \MR{1654165}

\bibitem[BM00]{brady2000three}
Thomas Brady and Jonathan~P McCammond, \emph{Three-generator {A}rtin groups of
  large type are biautomatic}, Journal of Pure and Applied Algebra \textbf{151}
  (2000), no.~1, 1--9.

\bibitem[BN74]{blass1974application}
Andreas Blass and Peter~M Neumann, \emph{An application of universal algebra in
  group theory.}, Michigan Mathematical Journal \textbf{21} (1974), no.~2,
  167--169.

\bibitem[Bod24]{Bodart2024}
Corentin Bodart, \emph{Membership problems in nilpotent groups}, Pre-print
  (2024), Available online at \texttt{arXiv:2401.15504}.

\bibitem[BP10]{Bell2010}
Paul~C. Bell and Igor Potapov, \emph{On the undecidability of the identity
  correspondence problem and its applications for word and matrix semigroups},
  Internat. J. Found. Comput. Sci. \textbf{21} (2010), no.~6, 963--978.

\bibitem[BS72]{brieskorn1972artin}
Egbert Brieskorn and Kyoji Saito, \emph{{A}rtin-gruppen und {C}oxeter-gruppen},
  Inventiones mathematicae \textbf{17} (1972), 245--271.

\bibitem[CCZ20]{Cadilhac2020}
Micha\"{e}l Cadilhac, Dmitry Chistikov, and Georg Zetzsche, \emph{{Rational
  Subsets of {B}aumslag-{S}olitar Groups}}, 47th International Colloquium on
  Automata, Languages, and Programming (ICALP 2020) (Dagstuhl, Germany) (Artur
  Czumaj, Anuj Dawar, and Emanuela Merelli, eds.), Leibniz International
  Proceedings in Informatics (LIPIcs), vol. 168, Schloss Dagstuhl --
  Leibniz-Zentrum f{\"u}r Informatik, 2020, pp.~116:1--116:16.

\bibitem[CGW09]{Crisp2009}
John Crisp, Eddy Godelle, and Bert Wiest, \emph{The conjugacy problem in
  subgroups of right-angled {A}rtin groups}, J. Topol. \textbf{2} (2009),
  no.~3, 442--460.

\bibitem[Cha92]{charney1992artin}
Ruth Charney, \emph{{A}rtin groups of finite type are biautomatic},
  Mathematische Annalen \textbf{292} (1992), no.~1, 671--683.

\bibitem[Cha95]{charney1995geodesic}
\bysame, \emph{Geodesic automation and growth functions for {A}rtin groups of
  finite type}, Mathematische Annalen \textbf{301} (1995), no.~1, 307--324.

\bibitem[Cha07]{Charney2007}
\bysame, \emph{An introduction to right-angled {A}rtin groups}, Geom. Dedicata
  \textbf{125} (2007), 141--158.

\bibitem[CP01]{Crisp2001}
John Crisp and Luis Paris, \emph{The solution to a conjecture of {T}its on the
  subgroup generated by the squares of the generators of an {A}rtin group},
  Invent. Math. \textbf{145} (2001), no.~1, 19--36.

\bibitem[CP05]{CrispParis2005}
John Crisp and Luis Paris, \emph{{A}rtin groups of type {B} and {D}}, Advances
  in Geometry \textbf{5} (2005), no.~4, 607--636.

\bibitem[Deh11]{Dehn1911}
M.~Dehn, \emph{\"{U}ber unendliche diskontinuierliche {G}ruppen}, Math. Ann.
  \textbf{71} (1911), no.~1, 116--144.

\bibitem[Deh97]{dehornoy1997fast}
Patrick Dehornoy, \emph{A fast method for comparing braids}, Advances in
  Mathematics \textbf{125} (1997), no.~2, 200--235.

\bibitem[Del72]{deligne1972immeubles}
Pierre Deligne, \emph{Les immeubles des groupes de tresses
  g{\'e}n{\'e}ralis{\'e}s}, Inventiones mathematicae \textbf{17} (1972),
  273--302.

\bibitem[DLS91]{Droms1991}
Carl Droms, Jacques Lewin, and Herman Servatius, \emph{Tree groups and the
  {$4$}-string pure braid group}, J. Pure Appl. Algebra \textbf{70} (1991),
  no.~3, 251--261.

\bibitem[Don23]{Dong2023}
Ruiwen Dong, \emph{Recent advances in algorithmic problems for semigroups}, ACM
  SIGLOG News \textbf{10} (2023), no.~4, 3--23.

\bibitem[EM94]{elrifai1994algorithms}
Elsayed~A Elrifai and Hugh~R Morton, \emph{Algorithms for positive braids}, The
  Quarterly Journal of Mathematics \textbf{45} (1994), no.~4, 479--497.

\bibitem[FGNB24]{foniqi2023membership}
Islam Foniqi, Robert~D. Gray, and Carl-Fredrik Nyberg-Brodda, \emph{Membership
  problems for positive one-relator groups and one-relation monoids}, Canadian
  Journal of Mathematics (to appear) (2024).

\bibitem[Gar69]{garside1969braid}
Frank~A Garside, \emph{The braid group and other groups}, The Quarterly Journal
  of Mathematics \textbf{20} (1969), no.~1, 235--254.

\bibitem[GGM10a]{GebhardtGonzalez2010b}
Volker Gebhardt and Juan Gonz\'{a}lez-Meneses, \emph{The cyclic sliding
  operation in {G}arside groups}, Math. Z. \textbf{265} (2010), no.~1, 85--114.
  \MR{2606950}

\bibitem[GGM10b]{GebhardtGonzalez2010}
\bysame, \emph{Solving the conjugacy problem in {G}arside groups by cyclic
  sliding}, J. Symbolic Comput. \textbf{45} (2010), no.~6, 629--656.
  \MR{2639308}

\bibitem[GKT02]{garber2002new}
David Garber, Shmuel Kaplan, and Mina Teicher, \emph{A new algorithm for
  solving the word problem in braid groups}, Advances in Mathematics
  \textbf{167} (2002), no.~1, 142--159.

\bibitem[GLR04]{Gordon2004}
C.McA. Gordon, D.D. Long, and A.W. Reid, \emph{Surface subgroups of {C}oxeter
  and {A}rtin groups}, Journal of Pure and Applied Algebra \textbf{189} (2004),
  no.~1, 135--148.

\bibitem[GMV14]{gonzalez2014twisted}
Juan Gonz{\'a}lez-Meneses and Enric Ventura, \emph{Twisted conjugacy in braid
  groups}, Israel journal of mathematics \textbf{201} (2014), 455--476.

\bibitem[GP12]{godelle2012basic}
Eddy Godelle and Luis Paris, \emph{Basic questions on {A}rtin-{T}its groups},
  Configuration Spaces: Geometry, Combinatorics and Topology, Springer, 2012,
  pp.~299--311.

\bibitem[Gra20]{gray2020undecidability}
Robert~D. Gray, \emph{Undecidability of the word problem for one-relator
  inverse monoids via right-angled artin subgroups of one-relator groups},
  Inventiones mathematicae \textbf{219} (2020), no.~3, 987--1008.

\bibitem[HH23]{haettel2023new}
Thomas Haettel and Jingyin Huang, \emph{New {G}arside structures and
  applications to {A}rtin groups}, arXiv preprint arXiv:2305.11622 (2023).

\bibitem[HM99]{Hermiller1999}
Susan~M. Hermiller and John Meier, \emph{{A}rtin groups, rewriting systems and
  three-manifolds}, J. Pure Appl. Algebra \textbf{136} (1999), no.~2, 141--156.

\bibitem[HO20]{huang2020large}
Jingyin Huang and Damian Osajda, \emph{Large-type {A}rtin groups are systolic},
  Proceedings of the London Mathematical Society \textbf{120} (2020), no.~1,
  95--123.

\bibitem[HR12]{holt2012artin}
Derek~F Holt and Sarah Rees, \emph{{A}rtin groups of large type are shortlex
  automatic with regular geodesics}, Proceedings of the London Mathematical
  Society \textbf{104} (2012), no.~3, 486--512.

\bibitem[HR13]{holt2013shortlex}
\bysame, \emph{Shortlex automaticity and geodesic regularity in {A}rtin
  groups}, Groups-Complexity-Cryptology \textbf{5} (2013), no.~1, 1--23.

\bibitem[HR15]{holt2015conjugacy}
\bysame, \emph{Conjugacy in {A}rtin groups of extra-large type}, Journal of
  Algebra \textbf{434} (2015), 12--26.

\bibitem[KP17]{KoPotapov2017}
Sang-Ki Ko and Igor Potapov, \emph{Composition problems for braids: Membership,
  identity and freeness}, Pre-print (2017), Available online at
  \texttt{arXiv:1707.08389}.

\bibitem[KSS07]{Kambites2007}
Mark Kambites, Pedro~V. Silva, and Benjamin Steinberg, \emph{On the rational
  subset problem for groups}, J. Algebra \textbf{309} (2007), no.~2, 622--639.

\bibitem[KT08]{kassel2008braid}
Christian Kassel and Vladimir Turaev, \emph{Braid groups}, vol. 247, Springer
  Science \& Business Media, 2008.

\bibitem[Lev15]{Gilbert2015}
Gilbert Levitt, \emph{Generalized {B}aumslag--{S}olitar groups: rank and finite
  index subgroups.}, Annales de l'Institut Fourier \textbf{65} (2015), no.~2,
  725--762.

\bibitem[Loh15]{Lohrey2015}
Markus Lohrey, \emph{The rational subset membership problem for groups: a
  survey}, Groups {S}t {A}ndrews 2013, London Math. Soc. Lecture Note Ser.,
  vol. 422, Cambridge Univ. Press, Cambridge, 2015, pp.~368--389.

\bibitem[Loh24]{lohrey2024membership}
M~Lohrey, \emph{Membership problems in infinite groups}, Conference on
  Computability in Europe, Springer, 2024, pp.~44--59.

\bibitem[LS77]{Lyndon1977}
Roger~C. Lyndon and Paul~E. Schupp, \emph{Combinatorial group theory},
  Ergebnisse der Mathematik und ihrer Grenzgebiete [Results in Mathematics and
  Related Areas], Band 89, Springer-Verlag, Berlin-New York, 1977.

\bibitem[LS08]{Lohrey2008}
Markus Lohrey and Benjamin Steinberg, \emph{The submonoid and rational subset
  membership problems for graph groups}, J. Algebra \textbf{320} (2008), no.~2,
  728--755.

\bibitem[Mak81]{Makanina1981}
T.~A. Makanina, \emph{The occurrence problem for the braid group
  {$\mathfrak{B}_{n+1}$} when {$n+1\geq 5$}}, Mat. Zametki \textbf{29} (1981),
  no.~1, 31--33, 154.

\bibitem[McC17]{mccammond2017mysterious}
Jon McCammond, \emph{The mysterious geometry of {A}rtin groups}, Winter Braids
  Lecture Notes \textbf{4} (2017), 1--30.

\bibitem[Mic99]{michel1999note}
Jean Michel, \emph{A note on words in braid monoids}, Journal of Algebra
  \textbf{215} (1999), no.~1, 366--377.

\bibitem[Mik58]{Mikhailova1958}
K.~A. Mikhailova, \emph{The occurrence problem for direct products of groups},
  Dokl. Akad. Nauk SSSR \textbf{119} (1958), 1103--1105.

\bibitem[MS17]{mccammond2017artin}
Jon McCammond and Robert Sulway, \emph{{A}rtin groups of {E}uclidean type},
  Inventiones mathematicae \textbf{210} (2017), no.~1, 231--282.

\bibitem[MV24]{minasyan2024right}
Ashot Minasyan and Motiejus Valiunas, \emph{Right-angled artin subgroups and
  free products in one-relator groups}, arXiv preprint arXiv:2404.15479 (2024).

\bibitem[NW01]{Niblo2001}
Graham~A. Niblo and Daniel~T. Wise, \emph{Subgroup separability, knot groups
  and graph manifolds}, Proc. Amer. Math. Soc. \textbf{129} (2001), no.~3,
  685--693.

\bibitem[Par02]{paris2002artin}
Luis Paris, \emph{{A}rtin monoids inject in their groups}, Commentarii
  Mathematici Helvetici \textbf{77} (2002), 609--637.

\bibitem[Pic01]{Picantin2001}
Matthieu Picantin, \emph{The conjugacy problem in small {G}aussian groups},
  Comm. Algebra \textbf{29} (2001), no.~3, 1021--1039. \MR{1842395}

\bibitem[Pot13]{Potapov2013Composition}
Igor Potapov, \emph{Composition problems for braids}, 33nd {I}nternational
  {C}onference on {F}oundations of {S}oftware {T}echnology and {T}heoretical
  {C}omputer {S}cience, LIPIcs. Leibniz Int. Proc. Inform., vol.~24, Schloss
  Dagstuhl. Leibniz-Zent. Inform., Wadern, 2013, pp.~175--187.

\bibitem[Rom22]{Romankov2022}
Vitaly Roman'kov, \emph{Undecidability of the submonoid membership problem for
  a sufficiently large finite direct power of the heisenberg group}, arXiv
  preprint arXiv:2209.14786 (2022).

\bibitem[RV96]{raptis1996subgroup}
E~Raptis and Dimitrios Varsos, \emph{On the subgroup separability of the
  fundamental group of a finite graph of groups}, Demonstratio Mathematica
  \textbf{29} (1996), no.~1, 43--52.

\bibitem[Seg05]{segal2005polycyclic}
Daniel Segal, \emph{Polycyclic groups}, no.~82, Cambridge University Press,
  2005.

\bibitem[Vog15]{Vogtmann2015}
Karen Vogtmann, \emph{{$GL(n,\Bbb Z)$}, {$Out(F_n)$} and everything in between:
  automorphism groups of {RAAG}s}, Groups {S}t {A}ndrews 2013, London Math.
  Soc. Lecture Note Ser., vol. 422, Cambridge Univ. Press, Cambridge, 2015,
  pp.~105--127. \MR{3495649}

\end{thebibliography}
\bibliographystyle{amsalpha}

\end{document}